\documentclass[12point]{article}

\usepackage{fullpage}
\usepackage{amsfonts,amsmath,epsf,epsfig,bbm}
\usepackage{amsthm}

\title{Tournaments, Johnson Graphs, and NC-Teaching} 

\author{Hans Ulrich Simon}
\newtheorem{theorem}{Theorem}[section]
\newtheorem{definition}[theorem]{Definition}
\newtheorem{lemma}[theorem]{Lemma}
\newtheorem{corollary}[theorem]{Corollary}
\newtheorem{remark}[theorem]{Remark}
\newtheorem{example}[theorem]{Example}
\newtheorem{claim}[theorem]{Claim}

\newcommand{\sm}{\setminus}
\newcommand{\eset}{\emptyset}
\newcommand{\cF}{{\mathcal F}}
\newcommand{\cK}{{\mathcal K}}
\newcommand{\cP}{{\mathcal P}}
\newcommand{\cX}{{\mathcal X}}
\newcommand{\cC}{{\mathcal C}}
\newcommand{\cT}{{\mathcal T}}
\newcommand{\cS}{{\mathcal S}}
\newcommand{\TD}{\mathrm{TD}}
\newcommand{\RTD}{\mathrm{RTD}}
\newcommand{\NCTD}{\mathrm{NCTD}}
\newcommand{\MTD}{\mathrm{M\mbox{-}TD}}
\newcommand{\ol}{\overline}
\newcommand{\seq}{\subseteq}
\newcommand{\ra}{\rightarrow}
\newcommand{\Span}[1]{\langle #1 \rangle}

\begin{document}

\maketitle

\begin{abstract}
Quite recently a teaching model, called ``No-Clash Teaching''
or simply ``NC-Teaching'', had been suggested that is provably 
optimal in the following strong sense. First, it satisfies 
Goldman and Matthias' collusion-freeness condition. Second, 
the NC-teaching dimension (= NCTD) is smaller than or equal 
to the teaching dimension with respect to any other collusion-free 
teaching model. It has also been shown that any concept class 
which has NC-teaching dimension $d$ and is defined over a domain 
of size $n$ can have at most $2^d \binom{n}{d}$ concepts. The main results
in this paper are as follows. First, we characterize the maximum
concept classes of NC-teaching dimension $1$ as classes which are 
induced by tournaments (= complete oriented graphs) in a very 
natural way. Second, we show that there exists 
a family $(\cC_n)_{n\ge1}$ of concept classes such that the well 
known recursive teaching dimension (= RTD) of $\cC_n$  grows 
logarithmically in $n = |\cC_n|$ while, for every $n\ge1$, 
the NC-teaching dimension of $\cC_n$ equals $1$. 
Since the recursive teaching dimension of a finite concept 
class $\cC$ is generally bounded $\log|\cC|$, the 
family $(\cC_n)_{n\ge1}$ separates RTD from NCTD in the most
striking way. The proof of existence of the family $(\cC_n)_{n\ge1}$
makes use of the probabilistic method and random tournaments.
Third, we improve the afore-mentioned upper bound $2^d\binom{n}{d}$
by a factor of order $\sqrt{d}$. The verification of the
superior  bound makes use of Johnson graphs and maximum subgraphs
not containing large narrow cliques.  
\end{abstract}

\section{Introduction}

Learning from examples that were carefully chosen by a teacher
(e.g.~a human expert) presents an alternative to the commonly
used model of learning from randomly chosen examples. A model
of teaching should be sufficiently restrictive to
rule out collusion between the learner and the teacher.
For instance, the teacher should not be allowed to encode
a direct representation of the target concept (such as a Boolean
formula or a neural network) within the chosen sequence of examples.
\cite{GM1996} suggested to consider a learner-teacher pair as
collusion-free if it satisfies the following condition:
if the learner is in favor of concept $C$ after having seen
the labeled teaching set $\cT$ chosen by the teacher, the learner
should again be in favor of $C$ after having seen a superset $\cS$
of $\cT$ as long as the label assignment in $\cS$ still coincides
with the label assignment induced by~$C$. In other words:
the learners guess $C$ for the target concept should not be
altered when the data give even more support to $C$ than
the original labeled teaching set $\cT$ is giving. 
Most existing abstract models of teaching are collusion-free 
in this sense.  Quite recently, \cite{KSZ2019} introduced 
a new model, called no-clash teaching or simply NC-teaching, 
that is collusion-free and furthermore optimal in the following
strong sense: \\
For any model $M$, let $\MTD(\cC)$ denote the corresponding
teaching dimension of concept class~$\cC$ (= smallest number
that upper-bounds the size of any of the employed teaching
sets provided that learner and teacher interact as prescribed
by model $M$). Then $\NCTD(\cC) \le \MTD(\cC)$ holds for any
model $M$ that satisfies Goldman and Mathias' collusion-freeness
criterion.

\smallskip\noindent
In this paper, we pursue the following questions:
\begin{itemize}
\item
What is the maximum size of a concept class which has NC-dimension $d$
and is defined over a domain of size $n$?
\item
How do the classes of maximum size look like?
\item
How does the NC-model of teaching relate to well known
model of recursive teaching?
\end{itemize}
Before we outline the structure of this paper, we put the
first two of these questions into a more general context.

\subsection{Bounds of the Sauer-Shelah Type}

A well-known lemma of Sauer~\cite{S1972} and Shelah~\cite{SH1972} 
states that a concept class of VC-dimension $d$ can induce at most
\[ \Phi_d(m) = \sum_{i=0}^{d}{m \choose d} \]
distinct binary label patterns on $m \ge d$ instances (taken
from the underlying domain). This implies that a concept class
of VC-dimension $d$ that is defined over a domain of size $n$ contains
at most $\Phi_d(n)$ distinct concepts. More results of the
Sauer-Shelah type are known in the literature. Here we focus
on results that are related to teaching. Consider, for instance,
the model of recursive teaching (introduced by~\cite{ZLHZ2011}).
As shown by~\cite{SSYZ2014}, $\Phi_d(n)$ also upper-bounds the
size of any concept class which has recursive teaching dimension~$d$
and is defined over a domain of size $n$. As shown by~\cite{KSZ2019},
$2^d{n \choose d}$ upper-bounds the size of any concept class
of NC-teaching dimension $d$ that is defined over a domain of size $n$.
While the upper bound $\Phi_d(n)$ is tight if $d$ equals
the VC-dimension or the recursive teaching dimension,
the corresponding bound $2^d{n \choose d}$ in case of
$d = \NCTD(\cC)$ is tight only for $d=1$ (as we will show
in this paper).

\subsection{Maximum Concept Classes}

Concept classes of VC-dimension $d$ inducing $\Phi_d(m)$
distinct binary label patterns on any sequence of $m$
distinct instances are called ``maximum classes'' (a
notion that dates back to early work of~\cite{W1987}).
The theoretical study of maximum classes had been fruitful
for several reasons:
\begin{itemize}
        \item Although there is a wide variety of maximum classes,
                they have much structure in common. This structure
                can be uncovered by exploiting the general definition
                of a maximum class (which abstracts away the
                peculiarities of specific maximum classes).
        \item The investigation of maximum classes and their
                structural properties often leads to problems
                with a combinatorial flavor that may be considered
                interesting in their own right.
        \item The validity of conjectures, believed to hold for
                arbitrary concept classes (like, for instance,
                the Sample-Compression conjecture of \cite{W2003})
                can be tested by showing their validity for
                maximum classes (as it has been done successfully
                by~\cite{FW1995} and~\cite{CCMW2019}).
\end{itemize}
Given the teaching-related bounds of the Sauer-Shelah type,
it is a natural idea to define and examine maximum classes 
in that context too. Here we are particularly interested in 
``NC-maximum classes'', i.e., concept classes of maximum size
among the ones having NC-dimension $d$ and being defined over 
a domain of size $n$.

\subsection{Structure of the Paper}

In Section~\ref{sec:facts}, we call into mind the definition
of various teaching models and the corresponding teaching dimensions. 
Section~\ref{sec:tournaments} contains the results which are
related to tournaments. We first define two concept classes,
the class  $\cC^1[G]$ of size $n$ and the class $\cC^2[G]$ of size $2n$, 
both of which are induced by a tournament~$G$ with $n$ vertices. 
Then we show the following results:
\begin{itemize}
\item
A concept class over domain $[n]$ is an NC-maximum class of 
NC-dimension $1$ if and only if there exists a tournament $G$ 
with $n$ vertices such that $\cC = \cC^2[G]$.
\item
For every tournament $G$, the the class $\cC^1[G]$ has NC-teaching 
dimension $1$.
\item
There is a strictly positive probability for the event that a random
tournament $G$ induces a class $\cC^1[G]$ whose recursive teaching 
dimension is at least $\log(n) - O(\log\log(n))$.
\end{itemize}
The last two results establish an RTD-NCTD ratio of order $\log n$. 
This is particularly remarkable since $\RTD(\cC)$ is upper-bounded 
by $\log|\cC|$ for every finite concept class $\cC$. 
In Section~\ref{sec:johnson-graphs}, one finds the results which are related
to Johnson-graphs. It is shown that a concept class which has NC-dimension $d$
and is defined over a domain of size $n$ contains 
at most $\left(2\sqrt{\frac{2}{d+1}} - \frac{2}{d+1}\right) \cdot 2^d\binom{n}{d}$
concepts. This improves the best previously known upper bound, $2^d\binom{n}{d}$,
by a factor of order~$\sqrt{d}$. It also shows that the size of an NC-maximum
class of NC-dimension $d \ge 2$ is strictly smaller than $2^d\binom{n}{d}$.
The key lemma behind this result is Lemma~\ref{lem:no-narrow-clique}, 
which relates the NC-teaching sets for a concept class $\cC$ to subgraphs 
of a Johnson graph which do not contain large narrow cliques. 
The final Section~\ref{sec:open-problems} mentions some open problems.

\section{Definitions, Notations and Facts} \label{sec:facts}

As usual a \emph{concept over domain $\cX$} is a function 
from $\cX$ to $\{0,1\}$ or, equivalently, a subset of~$\cX$.
A set whose elements are concepts over domain $\cX$ is referred to 
as a \emph{concept class over $\cX$}. The elements of $\cX$ are called 
\emph{instances}. The powerset of $\cX$ is denoted by $\cP(\cX)$. 
The set of all subsets of size $d$ of $\cX$ is denoted by $\cP_d(\cX)$. 
we refer to elements of $\cP_d(\cX)$ as \emph{$d$-subsets} of $\cX$.

%We begin this section by calling into mind 
%some notions that had been introduced in~\cite{KSZ2019}.

\begin{definition}[Teaching Models~\cite{GK1995,ZLHZ2011,KSZ2019}]
\label{def:teaching-models}
Let $\cC$ be a concept class over $\cX$.  
\begin{enumerate} 
\item
A \emph{teaching set for $C \in \cC$} is a subset $D \seq \cX$
which distinguishes $C$ from any other concept in $\cC$, i.e.,
for every $C' \in \cC\sm\{C\}$, there exists some $x \in D$
such that $C(x) \neq C'(x)$. The size of the smallest teaching set
for $C \in \cC$ is denoted by $\TD(C,\cC)$. The \emph{teaching 
dimension of $\cC$ in the Goldman-Kearns  model of teaching} is 
then given by
\[ 
\TD(\cC) = \max_{C \in \cC} |T(C,\cC)| \enspace . 
\]
A related quantity is
\[ 
\TD_{min}(\cC) = \min_{C \in \cC} |T(C,\cC)| \enspace . 
\]
\item
Let $T:\cC \ra \cP(\cX)$ be a mapping that assigns to every concept 
in $\cC$ a set of instances. 
$T$ is called \emph{admissable for $\cC$ in the NC-model\footnote{NC = No-Clash.} 
of teaching}, or simply an \emph{NC-teacher for $\cC$},  if, 
for every $C \neq C' \in \cC$, there exists $x \in T(C) \cup T(C')$
such that $C(x) \neq C'(c)$. The \emph{teaching dimension of $\cC$ 
in the NC-model of teaching} is given by
\[ \NCTD(\cC) = \min \{\max_{C \in \cC} |T(C)|: T\mbox{ is an NC-teacher for $\cC$}\} 
\enspace . 
\]
\item
Let $\cC_{min} \seq \cC$ be the easiest-to-teach concepts in $\cC$, i.e., 
\[
\cC_{min} = \{C \in \cC: \TD(C,\cC) = \TD_{min}(\cC)\} \enspace .
\]
The \emph{recursive teaching dimension of $\cC$} is then given by
\[ 
\RTD(\cC) = \left\{ \begin{array}{ll}
              \TD_{min}(\cC) & \mbox{if $\cC = \cC_{min}$} \\
              \max\{\TD_{min}(\cC) , \RTD(\cC\sm\cC_{min})\} &
              \mbox{otherwise}
            \end{array} \right.
\enspace . 
\]  
\end{enumerate}
In all three models, the set $T(C)$ is referred to as the
\emph{teaching set for $C$}. The \emph{order of $T$} is defined as 
the size of the largest of $T$'s teaching sets, i.e.,
order$(T) = \max_{C\in\cC}|T(C)|$.
\end{definition}

\noindent
Some remarks are in place here:
\begin{enumerate}
\item
It was shown in~\cite{DFSZ2014} that
\begin{equation} \label{eq:rtd-tdmin}
\RTD(\cC) = \max_{\cC' \seq \cC}\TD_{min}(\cC') \enspace . 
\end{equation}
\item
The set $T(C)$ in Definition~\ref{def:teaching-models} is an 
{\em unlabeled} set of instances. Intuitively, one should think 
of the learner as receiving the correctly \emph{labeled teaching set},
i.e.,the learner receives $T(C)$ {\em plus} the corresponding $C$-labels 
where $C$ is the concept that is to be taught. 
\item
We say that two concepts $C$ and $C'$ {\em clash} (with respect
to $T:\cC \ra \cP(\cX)$) if they agree on $T(C) \cup T(C')$, i.e,
if they assign the same $0,1$-label to all instances in $T(C) \cup T(C')$. 
NC-teachers for $\cC$ are teachers who avoid clashes between any pair 
of distinct concepts from $\cC$. 
\end{enumerate}

\noindent
As already observed by~\cite{KSZ2019}), NC-Teachers can be normalized:
\begin{itemize}
\item
We may assume without loss of generality that $|T(C)|=d$ 
for every $C\in\cC$ where $d$ denotes the order of $T$. 
\end{itemize}
This will be henceforth assumed. Let $n=|\cX|$ and $0 \le d \le n$. 
An NC-teacher $T$ for $\cC$ of order $d$ then assigns to every concept $C\in\cC$
a set taken from $\cP_d(\cX)$. Clearly $|\cP_d(\cX)| = {n \choose d}$ and,
for each fixed set $S\in\cP_d(\cX)$, there can be at most $2^d$ distinct
concepts in $\cC$ with NC-teaching set $S$. For this simple reason,
the following holds:

\begin{theorem}[\cite{KSZ2019}] \label{th:KSZ-bound}
Any concept class which has NC-teaching dimension $d$ and 
is defined over a domain of size $n$ contains at most $2^d{n \choose d}$
concepts.
\end{theorem}

An \emph{NC-maximum class (with respect to parameters $n$ and $d$)} 
is a concept class of largest size among all classes having NC-dimension $d$ 
and being defined over a domain $\cX$ of size $n$, say $\cX = [n]$. The size 
of such a class will be denoted by $M_{NC}(n,d)$ throughout this paper. 
According to Theorem~\ref{th:KSZ-bound}, $M_{NC}(n,d) \le 2^d\binom{n}{d}$.
We will see in the course of this paper that this upper bound
on $M_{NC}(n,d)$ is tight only for $d=1$ and, for $d\ge2$, 
it can be improved (at least) by a factor of order $\sqrt{d}$.

\section{Results on Concept Classes Induced by Tournaments} 
\label{sec:tournaments}

\noindent
The following notion will play a central role in this section:
\begin{definition}[concept class induced by a tournament] 
\label{def:tournament-class} 
        Let $G = ([n],E)$ be a \emph{tournament with $n$ players}, i.e.,
        $G$ is a directed graph obtained from the complete graph with vertices
        $1,\ldots,n$ by giving every edge $\{i,j\}$ an
        orientation (either $(i,j)$ or $(j,i)$).\footnote{Intuitively,
        edge $(i,j)$ represents the event that $j$ has lost against $i$
        in the tournament.} The concept classes $\cC^1[G]$ and $\cC^2[G]$
        are given by
        \[ 
        \cC^1[G] = \{\ol{C_1},\ldots,\ol{C_n}\} \   
        \mbox{ and }\  
        \cC^2[G] = \{C_1,\ldots,C_n\} \cup \{\ol{C_1},\ldots,\ol{C_n}\} 
        \]
        where, for $j=1,\ldots,n$, we set
        \begin{equation} \label{eq:G-concepts}
        C_j = \{i\in[n]: (i,j) \in E\}\ \mbox{ and }\ \ol{C_j} = [n] \sm C_j
        \enspace .
        \end{equation}
        We will refer to $\cC^1[G]$ (resp.~to $\cC^2[G]$) as \emph{the first
        (resp.~the second) concept class induced by~$G$}. 
\end{definition}
Intuitively, we may think of $C_j$ as consisting of all $i$
who have won against $j$ in the tournament~$G$. Note that $j \in \ol{C_j}$.

\begin{example} \label{ex:tournament}
        Consider the tournament $G_n$ with vertices $1,\ldots,n$ and,
        with edges $(i,j)$ for all $1 \le i < j \le n$ (i.e., edges
        are always directed from smaller to larger numbers).
        Let $\cC_n^2 = \cC^2[G_n]$ denote the second concept class 
        induced by $G_n$.
        Then the concepts in $\cC^2_n$ are $C_j = \{1,\ldots,j-1\}$
        and $\ol{C_j} = \{j,\ldots,n\}$ for $j=1,\ldots,n$.
        In other words, $\cC^2_n$ contains all left half-intervals over
        domain $[n]$ (including~$\eset$ but excluding $[n]$) and all
        right half-intervals (including $[n]$ but excluding~$\eset$).
\end{example}

We will show in Section~\ref{subsec:dim1} that a concept class $\cC$
over $[n]$ is an NC-maximum class of NC-teaching dimension $1$
if and only if there exists some tournament $G$ with $n$ players
such that $\cC = \cC^2[G]$. In Section~\ref{subsec:gap}, we show that 
there exists a family $(\cC_n)_{n\ge1}$ of concept classes 
such that $\RTD(\cC_n)$ grows logarithmically with $n = |\cC_n|$
while $\NCTD(\cC_n)=1$ for all $n\ge1$.
This is proven by the probabilistic method which deals here with
the concept class $\cC^1(G_n)$ for a random tournament $G_n$ with $n$
players.

\subsection{NC-Maximum Classes of Dimension 1} \label{subsec:dim1}

Here is the main result of this section:

\begin{theorem} \label{th:dim1}
        A concept class $\cC$ over $[n]$ is an NC-maximum class
        of NC-teaching dimension~$1$ if and only if $\cC = \cC^2[G]$
        for some tournament $G$ with $n$ players.
\end{theorem}

\begin{proof}
        Suppose first that $\cC = \cC^2[G]$ for some
        tournament $G=([n],E)$. Then $\cC^2[G]$ contains the $2n$ concepts
        that are given by~(\ref{eq:G-concepts}). Consider the
        mapping $T:\cC^2[G] \ra \cP_1([n])$ that assigns $\{j\}$ to the
        concepts $C_j$ and $\ol{C_j}$ for $j=1,\ldots,n$.
        $T$ avoids clashes between any pair of distinct concepts, which
        can be seen as follows:
        \begin{itemize}
                \item Since $j \notin C_j$ and $j\in\ol{C_j}$ holds
                        for every $j\in[n]$, there is no clash
                        between $C_j$ and~$\ol{C_j}$. 
        \end{itemize}
        Assume now that $i$ and $j$ are two arbitrary but distinct indices 
        from~$[n]$. The following observations show that neither $C_i$ 
        and $C_j$ nor $\ol{C_i}$ and $C_j$ agree on $\{i,j\}$:
        \begin{itemize}
                \item If $C_j$ agrees with $C_i$ on $\{i\}$, then $i \notin C_j$.
                        It follows that $(i,j) \notin E$. Thus $(j,i) \in E$ so
                        that $j \in C_i$, which means that $C_i$ disagrees
                        with $C_j$ on $\{j\}$. Hence there is no clash 
                        between  $C_i$ and $C_j$.
                \item By symmetry, it follows that there is no clash
                        between $\ol{C_i}$ and $\ol{C_j}$.
                \item If $C_j$ agrees with $\ol{C_i}$ on $\{i\}$,
                        then $i \in C_j$ so that $(i,j) \in E$. It follows
                        that $(j,i) \notin E$ so that $j \notin C_i$.
                        Thus $j\in\ol{C_i}$, which means that $\ol{C_i}$
                        disagrees with $C_j$ on $\{j\}$. Hence there is no clash
                        between $\ol{C_i}$ and $C_j$.
        \end{itemize}
        The above discussion shows that $T$ avoids clashes and is therefore
        an NC-teacher for $\cC^2[G]$. Since all teaching sets are of size $1$, 
        we get $\NCTD(\cC^2[G])=1$. In view of Theorem~\ref{th:KSZ-bound},
        the class $\cC^2[G]$ with $2n = 2^1{n \choose 1}$ concepts
        is an NC-maximum class. \\
        Suppose now that $\cC$ is an NC-maximum class of NC-teaching
        dimension $1$ over $[n]$. The first part of the proof
        in combination with Theorem~\ref{th:KSZ-bound} implies that $|\cC|=2n$.
        Let $T:\cC\ra\cP_1([n])$ be an NC-teacher for $\cC$.
        It follows that each set $\{j\}\in\cP_1([n])$ is assigned to exactly
        two concepts. Moreover these two concepts must disagree on $\{j\}$.
        We denote the concept with NC-teaching set $\{j\}$ that contains $j$
        (resp.~does not contain $j$) by $\ol{C_j}$ (resp.~by $C_j$).
        Fix two indices $i \neq j$ and consider the following assertions:
        \begin{enumerate}
                \item $C_j$ disagrees with $C_i$ on $\{i\}$.
                \item $C_j$ agrees with $\ol{C_i}$ on $\{i\}$.
                \item $\ol{C_i}$ disagrees with $C_j$ on $\{j\}$.
                \item $\ol{C_i}$ agrees with $\ol{C_j}$ on $\{j\}$.
                \item $\ol{C_j}$ disagrees with $\ol{C_i}$ on $\{i\}$.
                \item $\ol{C_j}$ agrees with $C_i$ on $\{i\}$.
                \item $C_i$ disagrees with $\ol{C_j}$ on $\{j\}$.
                \item $C_i$ agrees with $C_j$ on  $\{j\}$.
        \end{enumerate}
        Since the assignment of NC-teaching sets to concepts avoids clashes,
        it is easily seen that the following holds:
        \begin{itemize}
                \item Any assertion is an immediate logical consequence
                        of the preceding one.
                \item The first assertion is an immediate logical
                        consequence of the last one.
        \end{itemize}
        It follows that these eight assertions are equivalent.
        An inspection of the second and the fifth assertion
        reveals that $\ol{C_j} = [n] \sm C_j$.
        Consider now the directed graph $G=([n],E)$ with
        \[
        E = \{(i,j): \mbox{$C_j$ agrees with $\ol{C_i}$ on $\{i\}$}\}
        \enspace .
        \]
        An inspection of the second and the seventh assertion
        reveals that exactly one of the edges $(i,j)$ and $(j,i)$
        belongs to $E$. It follows that $G$ is a tournament.
        Moreover, the above definition of $E$ makes sure that,
        for every $j\in[n]$, $C_j = \{i\in[n]: (i,j) \in E\}$.
        We may therefore conclude that $\cC = \cC^2[G]$.
\end{proof}

\noindent
The following result, which shows that the general 
inequality $M_{NC}(n,d) \le 2^d\binom{n}{d}$ holds with equality for $d=1$, 
is a direct consequence of Theorem~\ref{th:dim1}:

\begin{corollary} \label{cor1:dim1}
        For all $n\ge1$: $M_{NC}(n,1)=2n$.
\end{corollary}

\noindent
Here is another direct consequence of Theorem~\ref{th:dim1}:

\begin{corollary} \label{cor2:dim1}
For $n\ge1$ and every tournament $G$ with $n$ players: $\NCTD(\cC^1[G]) = 1$.
\end{corollary}

\begin{proof}
Since $\cC^1[G] \seq \cC^2[G]$, a simple monotonicity argument shows 
that $\NCTD(\cC^1[G]) \le \NCTD[\cC^2[G]$. Clearly $\NCTD(\cC^1[G]) \ge 1$. 
Theorem~\ref{th:dim1} implies that $\NCTD(\cC^2[G]) = 1$. 
Thus $\NCTD(\cC^1[G]) = 1$.
\end{proof}

\subsection{Classes with NCTD 1 and Logarithmic RTD} \label{subsec:gap}

In this section, we make use of the following version of the Chernoff bound:

\begin{lemma}[\cite{C1952,KVaz1994}] \label{lem:chernoff}
Let $X_1,\ldots,X_m$ be a sequence of $m$ independent Bernoulli trials,
each with probability~$p$ of success. Let $Z = X_1 +\ldots+ X_m$ be the
random variable that counts the totel number of successes 
(so that ${\mathbb{E}}[S] = pm$). Then, for $0 \le \gamma \le 1$, 
the following holds:
\begin{equation} \label{eq:chernoff}
\Pr[Z < (1-\gamma)pm] \le \exp\left(\frac{-pm\gamma^2}{2}\right) \enspace .
\end{equation}
\end{lemma}

\noindent 
Here is the main result of this section:

\begin{theorem} \label{th:rtd-nctd-ratio}
For all sufficiently large $n$, there exists a concept class $\cC$ 
of size $n$ which satisfies 
$\RTD(\cC) \ge \TD_{min}(\cC) \ge \lfloor \log(n) - 2\log\log(2n) \rfloor - 4$ 
and $\NCTD(\cC) = 1$. 
\end{theorem}

\begin{proof}
We know from Corollary~\ref{cor2:dim1} that, for every tournament $G$ 
with $n$ players, the concept class $\cC^1[G]$ (which is of size $n$) 
has NC-dimension $1$. We know already from~(\ref{eq:rtd-tdmin})
that $\RTD$ is lower-bounded by $\TD_{min}$. The proof of the theorem can therefore
be accomplished by showing that, for all sufficiently large $n$, 
there exists a tournament $G$ with $n$ players 
such that $\TD_{min}(\cC^1[G]) \ge \lfloor \log(n) - 2\log\log(2n)\rfloor - 4$. 
For this purpose, we make use of the probabilistic method. Details follow. \\
Suppose that $\hat G = ([n],\hat E)$ is a \emph{random tournament},
i.e., for every $1 \le i < j \le n$, we decide by means of a fair coin
whether $(i,j)$ or $(j,i)$ is included into $\hat E$. Consider the
class $\cC^1[\hat G]$ (the first concept class induced by $\hat G$).
Let $k \ge 1$ be a parameter whose precise definition (as a function in $n$)
is postponed to a later stage. For every set $S \seq [n]$ of size $k$
and every $b: S \ra \{0,1\}$, let $Z_{S,b}$ be the random variable
which counts how many concepts $C \in \cC^1[\hat G]$ satisfy 
\begin{equation} \label{eq:prob-method}
\forall s \in S: C(s) = b(s) \enspace .
\end{equation}
Each concept $C_i$ with $i \in [n] \sm S$ satisfies
Condition~(\ref{eq:prob-method}) with a probability of exactly $2^{-k}$.
Therefore
\[ \mathbb{E}[Z_{S,b}] \ge 2^{-k} \cdot (n-k) \enspace . \]
An application of~(\ref{eq:chernoff}) with $p = 2^{-k}$, $\gamma = 1/2$ 
and $m = n-k$ yields that, for every fixed choice of $S$ and $b$, we have
\[ 
\Pr[Z_{S,b} < 2^{-(k+1)} (n-k)] \le \exp\left(\frac{-2^{-k}(n-k)}{8}\right)
= \exp\left(2^{-(k+3)}(n-k)\right) \enspace .
\]
As there are $\binom{n}{k}2^k$ possible choices for $(S,b)$, an application
of the union bound yields
\[
\Pr[\exists (S,b): Z_{S,b} < 2^{-(k+1)} (n-k)] \le 
\binom{n}{k}2^k \cdot \exp\left(-2^{-(k+3)}(n-k)\right) \enspace .
\]
We now set $k' = \log(n) - 2\log\log(2n) - 4$ and $k = \lfloor k' \rfloor$. 
\begin{claim} \label{cl:rtd-nctd-ratio}
For all sufficiently large $n$, we have
\[ 
 2^{-(k+1)}(n-k) \ge 2\ \mbox{ and } \binom{n}{k}2^k \cdot 
\exp\left(-2^{-(k+3)}(n-k)\right) < 1 \enspace .
\]
\end{claim}
\begin{description}
\item[Proof of the Claim:]
For almost all $n$, we have $n-k \ge n/2$. It therefore suffices to show 
that, for all sufficiently large $n$, we have
\[ 
2^{-(k+2)}n \ge 2\ \mbox{ and } 
\binom{n}{k}2^k \cdot \exp\left(-2^{-(k+4)}n\right) < 1 \enspace .
\]
The first inequality is valid because $k \le k' < \log(n)-3$.
After replacing $\binom{n}{k}$ by $n^k$ and taking the logarithm
on both hand-sides, we obtain the following sufficient condition
for the second inequality:
\[ k \log(2n) - 2^{-(k+4)}n < 0 \enspace . \] 
By the above choice of $k$, we have 
\[
k \log(2n) \le k' \log(2n) < \log^2(2n)\ \mbox{ and } \ 
2^{-(k+4)}n \ge 2^{-(k'+4)}n \ge \log^2(2n) \enspace ,
\] 
which completes the proof of the claim.
\end{description}
It follows that there is
a strictly positive probability for the event that, for each pair $(S,b)$,
at least $2$ concepts from $\cC^1[G]$ satisfy condition~(\ref{eq:prob-method}).
Since $S$ ranges over all subsets of $[n]$ of size $n$ and $b$ ranges
over all bit patterns from $\{0,1\}^k$, we can draw the following
conclusion: there exists a tournament $G$ with $n$ players such that
none of the concepts in $\cC := \cC^1[G]$ can be uniquely specified by $k$
(or less) labeled examples. This clearly implies that $\TD_{min}(\cC) \ge k$.
\end{proof}

With a little extra-effort, one can show that only an asymptotically
vanishing fraction of tournaments induces concept classes whose $\TD_{min}$
is upper bounded by $\lfloor \log(n) - 2\log\log(2n) \rfloor - 5$.
Hence almost all of these classes are hard to teach in the RTD-model.
More precisely, the following holds:

\begin{corollary} \label{cor:rtd-nctd-ration}
Let $\tau_n$ denote the fraction of tournaments $G = ([n],E)$ 
such that $\TD_{min}(\cC^1[G]) \le \lfloor \log(n) - 2\log\log(2n) \rfloor - 5$.
Then, for all sufficiently large $n$, 
we have that 
\begin{equation} \label{eq:tau-bound}
 \tau_n \le \frac{1}{(2n)^{\log(2n)}} \enspace .
\end{equation}
\end{corollary}

\begin{proof}
We use the probabilistic method thereby proceeding almost as in the proof 
of Theorem~\ref{th:rtd-nctd-ratio}. In the sequel, we stress the
differences to that proof: 
\begin{itemize}
\item
We set $k' = \log(n) - 2\log\log(2n) - 5$ and $k = \lfloor k' \rfloor$. 
\item
We have to show that\footnote{Compare with Claim~\ref{cl:rtd-nctd-ratio}.}
\begin{equation} \label{eq:rtd-nctd-ratio}
2^{-(k+2)}n \ge 2\ \mbox{ and }
\binom{n}{k}2^k \cdot \exp\left(-2^{-(k+4)}n\right) < \frac{1}{(2n)^{\log(2n)}}
\end{equation}
holds for all sufficiently large $n$.
\end{itemize}
The first inequality is immediate from the choice of $k$. As for the second 
inequality\footnote{Compare with the proof of Claim~\ref{cl:rtd-nctd-ratio}.}, 
it suffices to show that
\[ 
k \log(2n) - 2^{-(k+4)}n < \log\left(\frac{1}{(2n)^{\log(2n)}}\right) =
- \log^2(2n) \enspace . \]
By the above choice of $k$, we have
\[
k \log(2n) \le k' \log(2n) < \log^2(2n)\ \mbox{ and } \
2^{-(k+4)}n \ge 2^{-(k'+4)}n \ge 2\log^2(2n) \enspace ,
\]
which completes the proof of~(\ref{eq:rtd-nctd-ratio}). 
From these findings, it is easy to deduce~(\ref{eq:tau-bound}).\footnote{Compare
with the end of the proof of Theorem~\ref{th:rtd-nctd-ratio}.}
\end{proof}

\section{No-Clash Teaching Sets and their Relation to Johnson Graphs}
\label{sec:johnson-graphs}

In Section~\ref{subsec:johnson-graphs}, we call into mind the definition of Johnson
graphs (and related notions) along with some facts (all of which are well known
and also easy to verify). In Section~\ref{subsec:glb}, the tools from 
Section~\ref{subsec:johnson-graphs} are used to improve the upper 
bound $2^d\binom{n}{d}$ on $M_{NC}(n,d)$ by a factor of order $\sqrt{d}$.

\subsection{Johnson Graphs and their Subgraphs} 
\label{subsec:johnson-graphs}

\begin{definition}[Johnson graph]
        Let $J(n,k)$ denote the graph with vertex set $\cP_k([n])$
        and an edge between $A,B \in \cP_k([n])$ iff $|A \cap B| = k-1$.
        The graphs $J(n,k)$ with $1 \le k \le n$ are called
        {\em Johnson graphs}\footnote{named after the former
        American mathematician Selmer M.~Johnson}.
        A clique $\cK\seq\cP_k(n)$ in $J(n,k)$ is said to be
        {\em wide} if the sets in $\cK$ have a common intersection 
        of size $k-1$. Analogously, $\cK$ is said to be {\em narrow} if 
        the union of all sets in $\cK$ has size $k+1$.
\end{definition}

\begin{description}
\item[Warning:]
The distinction between wide and narrow cliques would be blured
if we represented the $N = \binom{n}{k}$ vertices simply by 
numbers $1,\ldots,N$. In what follows, the representation of
the $N$ vertices by $k$-subsets of $[n]$ is quite essential.
\end{description}

Note that $J(n,1)$ is isomorphic to the complete graph $K_n$.
The vertices $\{1\},\ldots,\{n\}$ of $J(n,1)$ form a wide clique.
$J(n,2)$ is isomorphic to the line graph $L(K_n)$.
$J(k,k)$ is a graph with a single vertex $[k]$ (and no edges).
$J(k+1,k)$ is isomorphic to $K_{k+1}$. The $k+1$ vertices
of $J(k+1,k)$ form a narrow clique. Cliques of size $2$ in $J(n,k)$
are wide and narrow. Cliques of size $3$ or more cannot be wide
and narrow at the same time. A clique of size $3$ is also called
{\em triangle} in the sequel. Here are some more of the known 
(and easy-to-check) facts concerning Johnson graphs:

\begin{lemma} \label{lem:johnson}
        \begin{enumerate}
                \item Distinct sets in $\cP_k([n])$ with a common
                        intersection of size $k-1$ (resp.~a union
                        of size $k+1$) must necessarily form a clique.
                \item Any clique in $J(n,k)$ is wide or narrow.
                \item The mapping $A \mapsto [n] \sm A$ is a graph
                        isomorphism between $J(n,k)$ and $J(n,n-k)$.
                        This isomorphism transforms narrow cliques
                        into wide cliques, and vice versa.
        \end{enumerate}
\end{lemma}

For any $\cF \seq \cP_k(n)$, we denote by $\Span{\cF}$ the subgraph
of $J(n,k)$ induced by $\cF$. We denote the subgraph relation by ``$\le$''
(e.g., $\Span{\cF} \le J(n,k)$). The following observation is rather 
obvious:
\begin{lemma} \label{lem:johnson-n2}
        \begin{enumerate}
                \item A graph with edge set $\cF \seq \cP_2([n])$ contains 
                      a triangle iff $\Span{\cF} \le J(n,2)$ contains 
                      a narrow triangle.
                \item A graph with edge set $\cF \seq \cP_2([n])$ contains 
                      a vertex of degree $c$ or more iff $\Span{\cF} \le J(n,2)$
                        contains a wide clique of size $c$.
        \end{enumerate}
\end{lemma}

We now fix some notation. The size of the largest $\cF \seq \cP_k(n)$
such that $\Span{\cF} \le J(n,k)$ does not contain a narrow $(t+1)$-clique is 
denoted by $H_t(n,k)$. 
Moreover, we set $h_t(n,k) = \binom{n}{k}^{-1} \cdot H_t(n,k)$.
For any $\cF \seq \cP_k([n])$ and $I\seq[n]$, we define
\[ \cF_{+I}  = \{J \in \cF | J \seq I\} \enspace . \]
Note that any (narrow or wide) clique in $\cF_{+I}$ would be a clique
of the same type and size within $\cF$. Hence, if $\Span{\cF}$ does
not contain a narrow $(t+1)$-clique, then the same holds for $\cF_{+I}$.
Given these notations and observations, the following holds:

\begin{lemma} \label{lem:gub}
For all $1 \le t \le k \le n-2$:
\begin{equation} \label{eq:gub}
  h(k,k) = 1\ \mbox{ and }\ h_t(n,k) \le h_t(n-1,k) \le h_t(k+1,k) = \frac{t}{k+1} 
\enspace .
\end{equation}
\end{lemma}

\begin{proof}
$J(k,k)$ is a graph consisting of a single isolated vertex and $J(k+1,k)$ 
is a narrow clique of size $k+1$. Hence $H_t(k,k)=1$ and $H_t(k+1,k)=t$,
which implies that $h_t(k,k) = 1$ and $h_t(k+1,k) = \frac{t}{k+1}$. 
The proof can now be accomplished by showing 
that $h_t(n,k) \le h_t(n-1,k)$.\footnote{This inequality is, in principle,
known from~\cite{KNS1964}. The proof in~\cite{KNS1964} is written in Hungarian
and it is formulated for hereditary properties of hypergraphs: 
if we view $\cF \seq \cP_k(n)$ as a set of $k$-uniform hyperedges, then not 
containing $t+1$ hyperedges whose union is of size $k+1$ will become a hereditary 
hypergraph property. In our application of this result, it is however more intuitive 
to view the elements of $\cF$ as vertices of the Johnson graph. In order to
make this paper more self-contained, we therefore included the short proof
for $h_t(n,k) \le h_t(n-1,k)$, which uses a simple averaging argument.}
Fix a family $\cF \seq \cP_k([n])$ of size $H_t(n,k)$ such 
that $\Span{\cF} \le J(n,k)$ does not contain a narrow $(t+1)$-clique..
There are $k \cdot H_t(n,k)$ occurrences of elements from $[n]$
within the sets of $\cF$. By the pigeon-hole principle,
there exists an $i \in [n]$ that occurs in
at most $\frac{k}{n} \cdot H_t(n,k)$ sets of $\cF$.
Set $I = [n]\sm\{i\}$. It follows that
\[
H_t(n-1,k) \ge |\cF_{+I}| \ge
\left(1-\frac{k}{n}\right)H_t(n,k) \enspace .
\]
Hence
\[
h_t(n,k) \le
\frac{n}{n-k} \binom{n}{k}^{-1} \binom{n-1}{k} h_t(n-1,k) =
h_t(n-1,k) \enspace .
\]
\end{proof}

We briefly note that the proof of the $h(n,k) \le h(n-1,k)$ made use 
only of the fact that the feature of avoiding a narrow $(t+1)$-clique
is inherited from $\cF$ to $\cF_{+I}$. Hence the same monotonicity
is valid whenever this kind of inheritance is granted.

\smallskip\noindent
The parameter $h_2(n,2)$ can be determined exactly:

\begin{remark} \label{rem:triangle-free}
For every $n \ge 2$, we have $h_2(n,2) \le \frac{n}{2(n-1)}$.
Moreover, this holds with equality if $n$ is even.
\end{remark}

\begin{proof}
According to Mantel's theorem~(\cite{M1907}) --- in a more general
form known as Turan's theorem~(\cite{T1941}) --- any triangle-free graph
has at most $n^2/4$ edges. For even $n$, this bound is tight because the 
complete bipartite graph $K_{n/2,n/2}$ is triangle free and has $n^2/4$ edges.
In combination with Lemma~\ref{lem:johnson-n2}, we may conclude
that $H_2(n,2) \le n^2/4$, and this holds with equality if $n$ is even. 
Hence $h_2(n,2) \le \binom{n}{2}^{-1} \cdot \frac{n^2}{4} = \frac{n}{2(n-1)}$,
again with equality if $n$ is even.
\end{proof}

\subsection{New Bounds on the Size of NC-Maximum Classes} 
\label{subsec:glb}

Let $\cC$ be a concept class over $[n]$ such that $\NCTD(\cC) = d$,
as witnessed by an NC-teacher $T: \cC \ra \cP_d(n)$.
Let $\cF = \{T(C): C \in \cC\} \seq \cP_d(n)$ be the family of all teaching 
sets assigned by $T$ to the concepts of $\cC$. For every $F \in \cF$, 
let $1 \le m(F) \le 2^d$ denote the number of concepts $C \in \cC$ 
with $T(C) = F$. Clearly $|\cC| = \sum_{F \in \cF}m(F)$. 
For every $2 \le t \le d$, we define
\begin{equation} \label{eq:heavy-teaching-sets}
\cF_t = \left\{F \in \cF: m(F) > \frac{2^{d+1}}{t+1}\right\} \enspace . 
\end{equation}
We view the sets in $\cF_t$ as vertices in the Johnson graph $J(n,d)$ 
so that $\Span{\cF_t}$ denotes the subgraph of $J(n,d)$ induced by $\cF_t$. 
With these notations, the following holds:

\begin{lemma} \label{lem:no-narrow-clique}
The graph $\Span{\cF_t} \le J(n,d)$ does not contain a narrow $(t+1)$-clique.
\end{lemma}

\begin{proof}
Assume for contradiction that $\Span{\cF_t}$ does contain 
a narrow $(t+1)$-clique $\cK$, say $\cK = \{F_1,\ldots,F_{t+1}\} \seq \cF_t$. 
Set $D = F_1 \cup\ldots\cup F_{t+1} \seq [n]$. The definition of a narrow
clique in $J(n,d)$ implies hat $|D| = d+1$. From the definition of $\cF_t$,
we may infer that $m(F_1) +\ldots+ m(F_{t+1}) > 2^{d+1}$. Thus $\cC$ contains 
more than $2^{d+1}$ concepts $C$ whose teaching set $T(C)$ belongs to $\cK$.
By the pidgeon-hole principle, there must be two distinct concepts $C_1$
and $C_2$ such that $T(C_1),T(C_2) \in \cK$ and $C_1$ and $C_2$ coincide
on $D$. But, since $T(C_1) \cup T(C_2) \seq D$, this means that $C_1$ 
and $C_2$ clash with respect to $T$. We arrived at a contradiction.
\end{proof}

We are now finally in the position to prove the (previously announced)
improved upper bound on $M_{NC}(n,d)$:

\begin{theorem} \label{th:gub}
For $2 \le t \le d \le n$, the following holds:
\[
M_{NC}(n,d) \le H_t(n,d)2^d + \left(\binom{n}{d}-H_t(n,d)\right)\frac{2^{d+1}}{t+1} 
= \left(h_t(n,d) + (1-h_t(n,d))\frac{2}{t+1}\right) \cdot 2^d \binom{n}{d}
\enspace .
\]
Moreover, for $t = \lfloor \sqrt{2(d+1)} \rfloor$, one gets
\begin{equation} \label{eq:improved-sauer-bound}
M_{NC}(n,d) \le \left(2 \sqrt{\frac{2}{d+1}} - \frac{2}{d+1}\right) \cdot 2^d \binom{n}{d}
\enspace . 
\end{equation}
\end{theorem}

\begin{proof}
Let $\cC$ be an NC-maximum class for parameters $n$ and $d$.
Then $|\cC| = M_{NC}(n,d)$.
Consider an NC-teacher $T:\cC\ra\cP_d(n)$ for~$\cC$. 
Let $\cF = \{T(C): C \in \cC\}$ and let $\cF_t \seq \cF$ be as 
defined in~(\ref{eq:heavy-teaching-sets}).
Lemma~\ref{lem:no-narrow-clique} implies that $\cF_t \seq \cP_d(n)$
is of size at most $H_t(n,d)$. The size of $\cC$ can therefore be
bounded as follows:
\[
|\cC| = \sum_{F \in \cF}m(F) =
\sum_{F \in \cF_t}m(F) + \sum_{F \notin \cF_t}m(F) \le
H_t(n,d)2^d + \left(\binom{n}{d}-H_t(n,d)\right) \frac{2^{d+1}}{t+1} 
\enspace .
\]
From this and $h_t(n,d) = \binom{n}{d}^{-1} \cdot H_t(n,d)$, we 
immediately obtain
\[ 
|C| \le \left(h_t(n,d) + (1-h_t(n,d))\frac{2}{t+1}\right) \cdot 2^d\binom{n}{d} 
\enspace .
\]
According to Lemma~\ref{lem:gub}, we have $h_t(n,d) \le \frac{t}{d+1}$.
Moreover, we may set  $t = \lfloor \sqrt{2(d+1)} \rfloor$ and can then
proceed as follows:
\[ 
h_t(n,d) + (1-h_t(n,d))\frac{2}{t+1} \le 
\frac{t}{d+1} + \left(1-\frac{t}{d+1}\right)\frac{2}{t+1} \le 
2 \sqrt{\frac{2}{d+1}} - \frac{2}{d+1}
\enspace .
\]
Putting everything together, we obtain~(\ref{eq:improved-sauer-bound}).
\end{proof}

\noindent
A simple computation shows that 
\[ 
2 \sqrt{\frac{2}{d+1}} - \frac{2}{d+1} \le 1
\] 
with equality for $d=1$ only. Hence the following holds:

\begin{corollary}
For $2 \le d \le n$, we have that $M_{NC}(n,d) < 2^d \binom{n}{d}$.
\end{corollary}

\noindent
For $d=2$, the upper bound on $M_{NC}(n,d)$ from Theorem~\ref{th:gub} 
can be slightly improved: 

\begin{corollary}
For every $n \ge 2$, the following holds:
\[ M_{NC}(n,2) \le \frac{(5n-4)n}{3} \approx \frac{5n^2}{3} \enspace . \]
\end{corollary}

\begin{proof}
We know from Theorem~\ref{th:gub} that
\[ 
M_{NC}(n,2) \le \left( h_2(n,2) + (1-h_2(n,2)) \frac{2}{3} \right) \cdot 4\binom{n}{2}
\enspace .
\]
We know from Remark~\ref{rem:triangle-free} that $h_2(n,2) \le \frac{n}{2(n-1)}$.
The assertion of the corollary now follows from a straightforward calculation.
\end{proof}

\noindent
Similar slight improvements of the bound in Theorem~\ref{th:gub}
are possible for other small values of $d$.

\section{Open Problems} \label{sec:open-problems}

According to Theorem~\ref{th:rtd-nctd-ratio}, there exists a family $(\cC_n)_{n\ge1}$
of concept classes such that $\RTD(\cC_n)$ grows logarithmically with $n = |\cC_n|$
while $\NCTD(\cC_n)=1$ for all $n \ge 1$. The existence proof is based on the
probabilistic method and therefore non-constructive. The best RTD-NCTD ratio, 
known so far for a concrete class, is the ratio for the class of parity functions 
in $n$ Boolean variables (a class of size $2^n$). It is shown in~\cite{FKSSZ2022}, 
that the RTD of the parity class equals $n$ while the NCTD of this class is bounded 
by $n/4$.

\begin{description}
\item[Open Problem 1:]
Find a concrete class which establishes a large RTD-NCTD ratio
(ideally a ratio of order $\log|\cC|$).
\end{description}

\medskip
Theorem~\ref{th:dim1} characterizes NC-maximum classes of NC-dimension $1$
as classes of the form $\cC^2[G]$ for some tournament $G$. Hence the structure
of NC-maximum classes of dimension $1$ is now perfectly known,  whereas 
the structure of NC-maximum classes of higher dimension is still unknown.

\begin{description}
\item[Open Problem 2:]
Find structural properties which are shared by all NC-maximum classes
of NC-dimension $d \ge 2$.
\end{description}

\medskip
An obstacle for solving the second open problem is that we do not even know
the size $M_{NC}(n,d)$ of NC-maximum classes having NC-dimension $d\ge 2$ and being
defined over a domain of size $n$. While we can conclude from 
Theorem~\ref{th:dim1} that $M_{NC}(n,1) = 2n$, the quantity $M_{NC}(n,d)$ with $d\ge2$
is still unknown to us (although Theorem~\ref{th:gub} makes a first step
towards finding non-trivial bounds on $M_{NC}(n,d)$ for $d \ge 2$).

\begin{description}
\item[Open Problem 3:]
Find better (upper and lower) bounds on $M_{NC}(n,d)$ respectively, if possible,
determine $M_{NC}(n,d)$ exactly.
\end{description} 

\paragraph{Acknowledgements.}

I want to thank Stasys Jukna for drawing my attention to narrow cliques 
in Johnson graphs and for many fruitful and inspiring discussions. 

%\bibliography{/Users/hanssimon/lernen}

%\appendix

\end{document}